\documentclass[reqno,oneside,11pt]{amsart}
\usepackage{graphicx}
\graphicspath{ {./graphics/} }
\usepackage[T1]{fontenc}
\usepackage[utf8]{inputenc}
\usepackage{lmodern}
\usepackage[french,english]{babel}
\usepackage{geometry} 

\usepackage{amsmath,amssymb,amsfonts,amsthm,mathtools}
\usepackage{fancyhdr}
\usepackage{emptypage}
\fancypagestyle{plain}{
	\fancyhf{}
	\fancyfoot[RO,RE]{\thepage}
	
	}

\usepackage{comment} 
\usepackage{xcolor}

\usepackage[noadjust]{cite}
\usepackage[
pagebackref=true,
colorlinks=true,
urlcolor=purple,
linkcolor=purple!87!black,
citecolor=green!60!black,
pdfborder={0 0 0}
]{hyperref}
\renewcommand*{\backref}[1]{}
\renewcommand*{\backrefalt}[4]{[{\tiny%
		\ifcase #1 Not cited.%
		\or Cited on page~#2.%
		\else Cited on pages #2.%
		\fi%
	}]}
\usepackage{bookmark}
\usepackage{dsfont}

\newcommand{\id}{\mathsf{id}}
\newcommand{\supp}[1]{\mathrm{supp}(#1)}

\newcommand{\Z}{\mathbb{Z}}
\newcommand{\R}{\mathbb{R}}

\newcommand{\E}{\mathbb{E}}

\newcommand{\tsp}{\mathrm{TSP}}

\newcommand{\thistheoremname}{}

\newtheorem*{genericthm*}{\thistheoremname}
\newenvironment{namedthm*}[1]
{\renewcommand{\thistheoremname}{#1}%
	\begin{genericthm*}}
	{\end{genericthm*}}
\theoremstyle{plain}
\newtheorem{thm}{Theorem}[section] 
\newtheorem{prop}[thm]{Proposition}
\newtheorem{lem}[thm]{Lemma}
\newtheorem{cor}[thm]{Corollary}
\newtheorem{rem}[thm]{Remark}
\newtheorem{defin}[thm]{Definition}
\newtheorem{question}[thm]{Question}

\newtheorem{claim}[thm]{Claim}

\title[{The CLT for lamplighters with an acyl. hyperbolic base}]{The CLT for lamplighter groups with an acylindrically hyperbolic base}
\author{Maksym Chaudkhari} 
\email[Maksym Chaudkhari]{mc637@usf.edu}
\address{Department of Mathematics and Statistics, University of South Florida, 4202 E Fowler Ave, CMC 342, Tampa, FL, 33620, USA}

\author{Christian Gorski} 
\email[Christian Gorski]{cgorski1@uw.edu}
\address{Department of Mathematics, University of Washington, Seattle, WA 98195, USA}

\author{Eduardo Silva} 
\email[Eduardo Silva]{eduardo.silva@uni-muenster.de, edosilvamuller@gmail.com}
\address{University of Münster, Einsteinstrasse 62, Münster 48149, Germany}
\urladdr{\url{https://edoasd.github.io/eduardo_silva_math/}}

\date{}

\begin{document}
	\begin{abstract} 
		We prove a Central Limit Theorem for the drift of a non-elementary random walk with a finite exponential moment on a wreath product $A\wr H\coloneqq \bigoplus_{H} A\rtimes H$ with $A$ a non-trivial finite group and $H$ a finitely generated acylindrically hyperbolic group. We also provide the upper bounds on the central moments of the drift. Furthermore, our results extend to the case where $A$ is an arbitrary (possibly infinite) finitely generated group.
		
	\end{abstract}
	\maketitle
\section{Introduction}
The classical Central Limit Theorem (CLT) states that if $(X_n)_{n\ge 1}$ is a collection of i.i.d. random variables with $\ell\coloneqq \E(X_1)<\infty$ and $\sigma^2\coloneqq \mathbb{V}(X_1)<\infty$, then the sequence \[\frac{\sum_{i =1}^{n}X_i-\ell n}{\sigma\sqrt{n}}, n\ge 1,\] converges in distribution to a standard Gaussian \cite{Lindeberg1922} (see also \cite[Chapter X.1]{Feller1968}). 
Extensions of this result have been studied intensely in the context of random walks on groups, a class of countable state Markov chains satisfying space-homogeneity.

More precisely, let $G$ be a finitely generated group endowed with a left-invariant metric $d$. Let $\mu$ be a probability measure on $G$ that is \emph{non-degenerate}: the support of $\mu$, $\text{supp}(\mu)$, generates $G$ as a semigroup. The \emph{$\mu$-random walk on $G$} is the sequence of random variables $Z_n=g_1\cdots g_n$, $n\ge 1$, where $(g_i)_{i\ge 1}$ are i.i.d. distributed according to $\mu$. If $\E^{\mu}\left[d(\id_G,Z_1)\right]<\infty$, then by Kingman's subadditive ergodic theorem, \cite{Kingman1968}, we have a well-defined \emph{asymptotic drift} $\ell\geq 0$, satisfying \[\lim_{n\to \infty}\frac{d(\id_G,Z_n)}{n}=\ell,\] almost surely and in $L^1$. One should compare this to the law of large numbers. 

The next natural line of investigation is with respect to the fluctuations of $d(\id_G,Z_n)$ around its mean. In the case $G=\mathbb{Z}^d$, this is addressed by the classical CLT. For other groups, however, one can observe a wide array of exotic behaviors for these fluctuations, which in general will strongly depend on the geometric properties of the group $G$.

There are finitely generated groups for which the normalized distances $d(\id_G,Z_n)/\sqrt{n}$ converge to a non-degenerate, non-Gaussian distribution (\cite[Proposition 6.2]{ErschlerZheng2022}), or for which it is not possible to normalize the differences $d(\id_G,Z_n)-\ell n$, $n\ge 1$, to obtain a convergence in law to a non-trivial distribution (see the remark after Theorem 8 in \cite{Bjorklund2010}).

Currently, all non-trivial examples of finitely generated groups known to satisfy a CLT for the distance to the identity with a Gaussian (or a combination of folded Gaussian) limiting distribution  are either finite extensions of abelian groups, which is close to the classical setting, or groups admitting sufficiently strong negative curvature properties, such as a proper action on a hyperbolic space.\footnote{By non-trivial we mean the groups that satisfy the CLT for a sufficiently large class of random walks and the CLT for them is not an immediate corollary of the CLT for other known examples.} Notably, this class of groups includes non-abelian free groups \cite{SawyerSteger1987, Ledrappier2001}, hyperbolic groups \cite{Bjorklund2010,BenoistQuint2016hyperbolic}, mapping class groups and $\mathrm{Out}(F_N)$ \cite{Horbez2018}. A more comprehensive list of groups that satisfy a CLT is given in Subsection \ref{subsection: background}. 

The main goal of the present paper is to establish the CLT for the distance to the identity in a new class of groups, namely wreath products of the form $A\wr H \coloneqq \bigoplus_H A\rtimes H$, where $\bigoplus_H A$ denotes the direct sum of isomorphic copies of a given finitely generated group $A$ indexed by a finitely generated acylindrically hyperbolic group $H$. The definition of the class of acylindrically hyperbolic groups is recalled in Definition \ref{def: ac hyp}, and it contains, in particular, free groups, discrete subgroups of $\text{Isom}(\mathbb{H}^n)$, and more generally all hyperbolic groups. The action of $H$ on $\bigoplus_H A$ by conjugation within $A\wr H$ corresponds to translation of the indices of the direct sum (see Subsection \ref{subsection: wreath products}). In the crucially important special case of a finite group $A$, the wreath product is also called a \emph{lamplighter group}.  This name is related to a classical interpretation, where the elements of $\bigoplus_H A$ are viewed as lamp configurations on $H$, while the $H$-component of each element of $A\wr H$ describes the position of a lamplighter who is changing states of individual lamps.
To simplify the exposition, we will work only with the lamplighter groups for the rest of the article until the last Subsection \ref{subsec: f.g. lamps}.
We call a metric on a lamplighter group $A\wr H$ \emph{standard} if it is the word metric associated with the generating set composed of a disjoint union of $A$  and a finite generating set of $H$ (see Subsection \ref{subsubsection: the word metric}).

\begin{thm}\label{CLT acylindrically hyperbolic}
	Let $A$ be a non-trivial finite group and let $H$ be a finitely generated acylindrically hyperbolic group. Let $d$ be a standard metric on $A\wr H$ and let $\mu$ be a symmetric probability measure on $A\wr H$ with a finite exponential moment. Suppose that the projection of $\mu$ to $H$  is a generating measure, or more generally, a non-elementary measure on $H$. Denote by $(Z_n)_{n\ge 0}$ the $\mu$-random walk on $A\wr H$ and by $\ell$ its asymptotic drift. Then there exists $\sigma>0$ such that the sequence of random variables $$ \left\{\frac{d(\id_G,Z_n) -\ell n}{\sigma \sqrt{n}} \right\}_{n \geq 1} $$  converges in distribution to a standard Gaussian. In addition, for any $p > 1$ there exists a constant $C=C(p,d,\mu)>0$ such that 
	\[\E^{\mu}\left|d(\id_G,Z_n) -\E^{\mu}d(\id_G,Z_n)\right|^p \leq Cn^{\frac{p}{2}}(\log n)^{5p}  \text{ for all } n\geq 1. \]
\end{thm}

Although $H$ is acylindrically hyperbolic and hence can be studied via its actions on hyperbolic spaces, the standard metric on $A\wr H$ behaves very differently from any hyperbolic metric. This can be seen from various geometric and combinatorial properties of the standard word metric on wreath products; we now describe an illustrative (though by no means exhaustive) list of examples.
\begin{enumerate}

	\item \label{item: examples 1} The \emph{divergence function} of a one-ended group is a quasi-isometry invariant that measures the minimal length of paths between two points avoiding a given ball in the Cayley graph. One-ended acylindrically hyperbolic groups always have superlinear divergence \cite[Theorem 2 and Proposition 4.1]{Sisto2016} (see also the discussion on p.~2503 of \cite{GruerSisto2018}), whereas wreath products always have linear divergence \cite[Theorem 1]{Issini2023}.
	
	\item In a wreath product, for every arbitrarily large radius, there exists an element that maximizes the distance to the identity within the ball of that radius around it (such elements are called \emph{dead ends of arbitrarily large depth}) \cite[Theorem A]{Silva2023}. These elements can be interpreted as regions of the group that exhibit positive curvature (see, e.g., \cite[Theorem 2.7]{kropholler2020} and \cite{BN}). This property is not satisfied by any word metric on a hyperbolic group \cite{Bogopolski}.
	
	\item In \cite{Parry1992} it is shown that the growth series of balls in groups of the form $A\wr F_k$, where $A$ is finite and non-trivial, and $F_k$ is a free group of rank $k\ge 2$ is not a rational function. This contrasts with the case of hyperbolic groups, where the growth series is always rational \cite[Theorem 3.4.5]{Epstein1992} (see also \cite{Cannon1984}).
\end{enumerate}

Point \eqref{item: examples 1} is key here. In proofs of the CLT for random walks on groups with hyperbolic properties, one uses divergence of the group to argue that the geodesic connecting $\id_G$ to $Z_{2n}$ must pass through a controlled neighbourhood of $Z_n$ with good probability, from which one can approximate $d(\id_G, Z_{2n})$ by a sum of independent random variables. This is the perspective taken in \cite{ChawlaCLT}. As wreath products have linear divergence, this ingredient is missing, so one must use deeper properties of the word metric.

As we explain in Subsection \ref{subsubsection: the word metric}, the values of the standard word metric on $A\wr H$ can be expressed in terms of solutions of the Traveling Salesman Problem (TSP) on the Cayley graph of $H$. Namely, the word length of an element can be expressed as the minimal length of a path on the Cayley graph of $H$ that starts at the identity $\id_H$, visits all positions where there are lamps lit, and which finishes its journey in some determined position of $H$.

In the proof of Theorem \ref{CLT acylindrically hyperbolic}, we leverage the geodesic tracking results  of Mathieu and Sisto \cite{MathieuSisto2020} to bound the defect of the $\tsp$ along a trajectory of a random walk in $H$. In our setting, this defect roughly quantifies the discrepancy between the sum of optimal solutions to two TSP problems versus the optimal solution to their concatenation. The importance of this quantity is explained by the connection between the standard metric on the Cayley graph of $A\wr H$ with the $\tsp$ in the Cayley graph of $H$, outlined in Subsection \ref{subsubsection: the word metric}. Here, the locations which must be visited in the TSPs are (roughly) the lamps which have been lit during the course of the random walk over different time intervals---times in $[0,m]$, $[m,n+m]$, and $[0,n+m]$ respectively. Since random walks on acylindrically hyperbolic groups track geodesics, all lamps lit during a given time interval lie, with high probability, close to a geodesic from the origin to the lamplighter’s final position in $H$. Since the random walk has linear progress, the lamps lit from times $[0,m]$ and from times $[m,n+m]$ do not have too much spatial overlap with high probability. A deterministic argument then allows us to quantitatively bound the defect in terms of the distance to the geodesic and the overlap between segments. The defect bounds we obtain are not uniform in $m,n$ as they are in \cite{MathieuSisto2020}, but they grow slowly in $m,n$, and a slight modification of the argument in \cite{MathieuSisto2020} gives the same conclusion. The results of Mathieu and Sisto, along with the variants relevant to our setting, are recalled in Section \ref{section: MS}. These are then combined, in Section \ref{section: final proof CLT}, with deterministic combinatorial estimates for the TSP established in Section \ref{section: combinatorial}, to prove Theorem \ref{CLT acylindrically hyperbolic}. Moreover, with some additional technical modifications, our argument allows us to obtain the CLT for the random walks on the more general wreath products $A\wr H$ where $A$ is an arbitrary finitely generated group. We explain these modifications in the Subsection \ref{subsec: f.g. lamps}, see Corollary \ref{cor: CLT for general wr product}.

\subsection{Background}\label{subsection: background}
The CLT for the distance to the identity on non-abelian free groups is due to Sawyer and Steger \cite{SawyerSteger1987} and Ledrappier \cite{Ledrappier2001} for finitely supported step distributions. This was later extended to non-elementary hyperbolic groups by Björklund \cite{Bjorklund2010} for the Green metric under a finite exponential moment assumption, and subsequently by Benoist and Quint \cite{BenoistQuint2016hyperbolic} to general word metrics under a finite second moment assumption. Mathieu and Sisto \cite{MathieuSisto2020} further extended these results to a larger class of groups, including all acylindrically hyperbolic groups. Additional classes of groups that satisfy a version of the CLT are linear groups by Benoist and Quint \cite{BenoistQuint2016}, mapping class groups and $\mathrm{Out}(F_n)$ by Horbez \cite{Horbez2018}, groups that act on $\mathrm{CAT}(0)$-spaces with contracting isometries by Le Bars \cite{LeBars2022}, groups with contracting elements by Choi \cite{Choi2023}, groups that contain superlinear divergent geodesics by Chawla, Choi, He and Rafi \cite{ChawlaCLT}, and groups of affine transformations of a horospherical product of hyperbolic spaces by Bahmanian, Forghani, Gekhtman and Mallahi-Karai \cite{BahmanianForghaniGekhtmanMallahiKarai2024}.

Wreath products are a classical construction in group theory that have provided many examples of interesting probabilistic phenomena in random walks on groups. One of the first such examples is due to Kaimanovich and Vershik \cite{KaimanovcihVershik1983}, who showed that the lamplighter group $\Z/2\Z\wr \Z$ has exponential volume growth but that every random walk on it with a symmetric, finitely supported step distribution has zero asymptotic drift. Before their work, it was known that positive asymptotic drift implies exponential growth, but the converse was open. In contrast, they proved that any non-degenerate symmetric random walk with finitely supported steps on $\Z/2\Z\wr \Z^d$ for $d\ge 3$ has a non-trivial Poisson boundary, thus providing the first examples of an amenable group with this property. Varopoulos used the wreath product $\Z/2\Z\wr \Z$ to provide the first example of a group for which the return probability to the origin of a simple random walk decays asymptotically as $\mathrm{exp}(-n^{1/3})$ \cite{Varopoulos1983,Varopoulospotential}. This property was further studied by Pittet and Saloff-Coste \cite{PittetSaloffCoste2002}, who showed that for any fixed $0<\alpha<1$ there are wreath products of the form $G\wr \Z$ for which the return probability to the origin of a simple random walk decays asymptotically faster than $\mathrm{exp}(-n^{\alpha})$. Erschler studied the asymptotic growth of the expected distance to the identity on such wreath products $G\wr \Z$, showing in particular, that the drift of a random walk can grow at an intermediate rate between $\sqrt{n}$ and $n$ \cite{Erschlerasymptotics,Erschlerdrift}, answering a question of Vershik. 

More recently, Erschler and Zheng studied random walks on $\mathbb{Z}/2\mathbb{Z}\wr \mathbb{Z}^2$ with a finite $(2+\varepsilon)$-moment for some $\varepsilon>0$, and proved that $d(\id,Z_n)$, normalized by its expectation, converges almost surely to a constant \cite[Theorem 1.1]{ErschlerZheng2022}. They also proved in \cite[Proposition 6.2]{ErschlerZheng2022} that for any non-trivial finite group $A$, the drift of a simple random walk on $A\wr \mathbb{Z}$ satisfies the following version of a central limit theorem: there are constants $\sigma, c_1,c_2>0$ such that the sequence of random variables $d(\id,Z_n)/(\sigma \sqrt{n})$ converges in law to the distribution $c_1(\mathcal{R}_1-|B_1|)+c_2|B_1|$, where $(B_t)_{t\ge 0}$ is a standard Brownian motion and $\mathcal{R}_t$ denotes the size of its range at time $t$. Regarding other limit laws, Revelle established a law of the iterated logarithm on certain wreath products $G\wr \Z$ \cite{Revelle20032}. Our main result (Theorem \ref{CLT acylindrically hyperbolic}) is the first to establish a CLT for random walks on wreath products whose base group is not $\Z$.

One cannot expect a version of the CLT to hold for random walks on every finitely generated group. Indeed, as Björklund remarked in the paragraph following Theorem 8 in \cite{Bjorklund2010}, the distance to the identity in a simple random walk on $(\Z\wr\Z)\times F_2$ exhibits fluctuations of order $n^{1/2}$ in directions governed by the $F_2$ factor and of order $n^{3/4}$ in those governed by the $\Z\wr\Z$ factor. Moreover, Erschler and Zheng constructed finitely generated groups $G$ for which $d(\id,Z_n)$, normalized by its expectation, converges in law to different limit distributions along distinct subsequences \cite[Proposition 6.6]{ErschlerZheng2022}.

\subsection{Organization}
In Section \ref{section: preliminaries} we recall several basic definitions and explain the terminology used throughout the article. Next, in Section \ref{section: MS}, we review the results of Mathieu and Sisto established in \cite{MathieuSisto2020} and explain how to generalize them to the case where the defect has slowly growing moments.

In Section \ref{section: combinatorial}, we prove the main deterministic geometric lemma for the solutions of the $\mathrm{TSP}$, see Lemma \ref{comb_lemma_1d}. Finally, in Section \ref{section: final proof CLT} we provide the probabilistic argument and finish the proof of the main results of the article. Since the case of the finite group of lamps $A$ already contains the main ideas behind the proof, to simplify the exposition, we first establish the CLT and the upper bounds on the central moments of the drift for the lamplighter groups and prove Theorem \ref{CLT acylindrically hyperbolic}. The modifications required in the general case of a finitely generated group $A$ are outlined in Subsection \ref{subsec: f.g. lamps}, and in particular, in the proof of Corollary \ref{cor: CLT for general wr product}.

\subsection{Acknowledgments}
This work started during the thematic program ``Randomness and Geometry'' at the Fields Institute for
Research in Mathematical Sciences in 2024. The authors thank the Fields Institute and the organizers of the program for their hospitality. Maksym Chaudkhari and Christian Gorski were also supported by the Fields Postdoctoral Fellowships during their stay at the Fields Institute. Eduardo Silva is funded by the Deutsche Forschungsgemeinschaft (DFG, German Research Foundation) under Germany's Excellence Strategy EXC 2044 –390685587, Mathematics Münster: Dynamics–Geometry–Structure.  We thank Kunal Chawla for his participation in the first stages of the project, and in particular, for the discussions on the results of Mathieu and Sisto and  Gou{\"e}zel, for the exploration of an approach relying on the pivoting technique in the case of a hyperbolic base group, and for the feedback on a preliminary draft of this work. We also thank Alessandro Sisto, Giulio Tiozzo, and Tianyi Zheng for helpful discussions during the program at the Fields Institute.

\section{Preliminaries}\label{section: preliminaries}

\subsection{Metric spaces and paths}
In this article, we will always assume that we work with a proper geodesic metric space $(X,d)$. 
A path $\gamma \colon [a,b] \rightarrow X$ is a continuous map from the segment $[a,b] \subset \mathbb{R}$ to $X$. The path $\gamma$ is endowed with the natural orientation and we refer to $\gamma(a)$ as a starting point  and $\gamma(b)$ as the ending point of $\gamma$. 
If $\gamma$ consists of a single point, this point is considered to be both the ending point and the starting point of $\gamma$.

\begin{defin}
	A path $\gamma \colon [a,b] \rightarrow X$ is called a geodesic segment connecting $\gamma(a)$ and $\gamma(b)$ if it is an isometric embedding of $[a,b]$ into $X$. The length of a geodesic segment will be denoted by $|\gamma|=d(\gamma(a), \gamma(b))=|b-a|$.
\end{defin}
Let us also recall the notion of piecewise geodesic and its length.
\begin{defin}\label{def: piecewise-geodesic}
	We call a path $\alpha \colon[a,b] \rightarrow X$ a piecewise geodesic if it can be represented as the concatenation of finitely many geodesic segments. If $\gamma_1, \ldots, \gamma_k$ is the list of these segments, then the length of $\alpha$ is defined as $|\alpha|=\sum_{i=1}^{k}|\gamma_i|$.
\end{defin}

It is easy to see that the length of the piecewise geodesic does not depend on the choice of representation.

To simplify the notation, we will extend the length notation to lists of geodesic segments.
For any two points $P, Q \in X$ we will denote by $PQ$ a geodesic segment joining them, and the length of this segment will be denoted by $|PQ|=d(P,Q)$. Moreover, for any finite list of geodesic segments $\theta$, we will denote by $|\theta|$ its total length computed as the sum of the lengths of its entries. Notice that if the list has repetitions, we also repeat the corresponding lengths in the sum.

Finally, we will also use the the following definition of a projection.
\begin{defin}
	For a point $P \in X$ and geodesic segment $\gamma$ we will denote by $\pi_{\gamma}(P)$ the set of all points on $\gamma$ that minimize the distance to $P$.    
\end{defin}

\subsection{The Traveling Salesman Problem.} We will work with a version of the Traveling Salesman Problem with fixed starting and ending points.

\begin{defin}
	For the points $A, B \in X$ and a finite set of points $L \subset X$, the $\mathrm{TSP}(A,L,B)$ denotes the length of the shortest piecewise geodesic path in $X$ which starts at $A$, ends at $B$, and visits every point in $L$. We will call a \textit{solution of} $\mathrm{TSP}(A,L,B)$ any path $\alpha$ of minimal length that satisfies these conditions.  
\end{defin}

It is easy to see, that to find a solution of the $\tsp(A, L,B)$ we only need to consider piecewise geodesic paths $\alpha$ constructed by connecting the points from $L \cup \{A,B\}$  by geodesic segments.

Whenever $\alpha$ is a path of this form, we can list the points of $L=\{l_1,\ldots,l_k\}$ in the order of their first appearance along $\alpha$,  $L=(l_{\pi(1)},\ldots,l_{\pi(k)})$. Then the length of $\alpha$, denoted by $|\alpha|$, is computed as  $$|\alpha|=
d(A,l_{\pi(1)})+\sum_{i=1}^{k-1}d(l_{\pi(i)},l_{\pi(i+1)})+d(l_{\pi(k)},B).$$  

Notice that $\pi \in \mathrm{Sym}(k) $ determines $\alpha$ uniquely up to the choice of the geodesic segments connecting $A$ with $l_{\pi(1)}$, $B$  with $l_{\pi(k)}$, and $l_{\pi(i)}$ with $l_{\pi(i+1)}$ for $i=1, \ldots, k-1$. Therefore, solving the TSP is equivalent to finding a permutation that minimizes the sum above. In particular, the shortest path always exists in this setting.

For any solution $\alpha$ of $\mathrm{TSP}(A,L,B)$, we will refer to the points in $\{A,B\}\cup L$ as the \textit{nodes} of $\alpha$.

\subsection{Acylindrically hyperbolic groups}
We recall the definition and notable examples of the acylindrically hyperbolic groups, and we refer the reader to \cite{Osin2016} for a detailed treatment of the subject.

Recall that a geodesic metric space $(X,d)$ is called \textit{hyperbolic}, if there exists $\delta > 0 $ such that for every geodesic triangle $T$ in $X$, 
every side of $T$ is contained in the $\delta$-neighborhood of the union of the other two sides.

Let $G$ be a group that acts on a metric space $(X,d)$ by isometries. This action is called \textit{acylindrical} if for every $r \geq 0$, there exist $N,\, R > 0$ such that whenever $x,y \in X$ satisfy $d(x,y)>R$, there are at most $N$ elements $g \in G$ that satisfy both inequalities $d(x,gx)<r$ and $d(y,gy)<r$. The action is called \textit{non-elementary} if $G$ is not virtually cyclic and the $G$-orbits are unbounded.

\begin{defin}\label{def: ac hyp}
	A group $G$ is called acylindrically hyperbolic if it admits a non-elementary acylindrical action on a hyperbolic geodesic metric space $(X, d)$ by isometries. 
\end{defin}

Examples of acylindrically hyperbolic groups include all non-elementary hyperbolic groups, all non-elementary relatively hyperbolic groups, $Out(F_n), n\geq 2 $, the mapping class groups $MCG(\Sigma_{g,p})$ of connected oriented surfaces of genus $g \geq 0$ with $p \geq 0$ punctures, except for the case when $g=0$ and $p \leq 3$, and many fundamental groups of compact $3$-manifolds (see \cite{MinasyanOsin2015}).

The main reason why we consider the base group $H$ to be acylindrically hyperbolic in the statement of Theorem \ref{CLT acylindrically hyperbolic} is to be able to use Proposition \ref{prop: tracking} in Section \ref{section: final proof CLT}, which gives quantitative estimates on how a random walk on an acylindrically hyperbolic group tracks geodesic segments.

\subsection{Wreath products}\label{subsection: wreath products}

We consider the restricted wreath product $G=A\wr H$, where $A$ is a non-trivial finite group and $H$ is a finitely generated group. More explicitly, $G=\bigoplus_H A \rtimes H$, where the direct sum $\bigoplus_H A$ is viewed as the group of functions $f: H \rightarrow A$ with finite support, and $H$ acts on $\bigoplus_H A$ by left translation, so for all $h, x \in H$ and all $ f \in \bigoplus_H A $, we have $h.f(x)=f(h^{-1}x).$
If $g_1=(f_1,h_1)$ and $g_2=(f_2,h_2)$ are two elements of $G=A\wr H$, their product $g_1g_2$ can written explicitly as $g_1g_2=(f_1\cdot (h_1.f_2),h_1\cdot h_2)$.

\subsubsection{The standard word metric on wreath products}\label{subsubsection: the word metric}
Let $S_H$ be a finite symmetric generating set of $H$. 
The standard generating set $S$ of $G=A \wr H$ associated with $S_H$ is given by  

\[ S \coloneqq  \Big\{(\delta_a, \id_H), \, (\mathbf{0}, s) \Big| a \in A, \; \text{and}  \;  s \in S_H \Big\},
\]

where $\delta_a (x) =a$ iff $x =\id_H$,   and otherwise $\delta_a (x) =\id_A $, and $\mathbf{0} (x) =\id_A$ for any $x \in H$.

It is well-known that the word length of any $g=(f,h) \in G $ with respect to $S$ can be computed as follows:
\[|g|_{\mathrm{S}}=\mathrm{TSP}(\id_H, \supp{f},h)  +|\supp{f}|,\]

where the $\tsp$ is solved with respect to the distance $d_H$ on the Cayley graph of $H$ induced by the word metric corresponding to $S_H$, see, e.g., \cite[Theorem 1.2]{Parry1992}.

\subsection{Random walks on groups}\label{subsection: random walks}

Let $G$ be a countable group and consider a probability measure $\mu$ on $G$. Consider the product space $\Omega\coloneqq G^{\Z_{+}}$ endowed with the product $\sigma$-field. For each $n\ge 1$ we denote by
\begin{equation*}
	\begin{aligned}
		X_n:\Omega&\to G\\
		w\coloneqq(w_1,w_2,\cdots)&\mapsto X_n(w)\coloneqq w_n
	\end{aligned}
\end{equation*}
the $n$-th coordinate map. We endow $\Omega$ with the product probability measure $\mu^{\Z_{+}}.$

We denote by
\begin{equation*}
	\begin{aligned}
		\theta:\Omega&\to \Omega\\
		w\coloneqq(w_1,w_2,\cdots)&\mapsto \theta(w)\coloneqq (w_2,w_3,\ldots)
	\end{aligned}
\end{equation*}
the shift map in the space of increments.

Now we define the $\mu$-random walk $\{Z_n\}_{n\ge 0}$ on $G$ as follows. We define $Z_0(w)=\id_G$ for each $w\in \Omega$, and for each $n\ge 1$ we set
\[
Z_n(w)\coloneqq Z_{n-1}(w)\cdot X_n(w).
\]
We remark that $Z_n(w) (Z_m\circ \theta^n)(w)=Z_{n+m}(w)$, for each $w\in \Omega$ and $n,m\ge 1$.

When the group $G$ is endowed with a left-invariant metric $d$, we say that $\mu$ has finite exponential moment with respect to $d$ if there exists $\alpha >0 $ such that $$e(\mu,d) \coloneqq \E^{\mu}\exp(\alpha d(\id,X_1)) < \infty.$$ 

The reader may find a more detailed discussion of non-elementary measures on acylindrically hyperbolic groups, for example, in \cite[Section 8 and 9]{MathieuSisto2020}. We will simply state the definition below and note that for an acylindrically hyperbolic group $G$, every measure $\mu$ whose support $\supp{\mu}$ generates $G$ as a semigroup is non-elementary (see, for example \cite[Remark 9.2]{MathieuSisto2020}).
\begin{defin}
	Let $G$ be an acylindrically hyperbolic group and assume that its acylindrical hyperbolicity is witnessed by the action of $G$ on a geodesic Gromov hyperbolic space $X$. We will call the measure $\mu$ \textit{non-elementary}, if the  subsemigroup of $G$ generated by the support of $\mu$ contains two loxodromic elements that freely generate a free group. 
\end{defin}
For the definition of loxodromic elements, see for example, \cite[page 2]{Osin2016}.

Finally, if $G=A \wr H$, then the positions of the random walk can be described explicitly as follows, if  $Z_n =(F_n,g_n)$  and $X_n=(f_n,h_n)$ for $n \geq 1$, then we have $$Z_{n+1}=(F_{n+1}, g_{n+1})=\left(F_{n} \cdot (g_{n}.f_{n+1}), g_{n}\cdot h_{n+1}\right).$$
\section{Mathieu-Sisto deviation inequalities and consequences}\label{section: MS}

Let $G$ be a countable group equipped with a probability measure $\mu$.\footnote{The results of Mathieu and Sisto that we describe below also hold in greater generality, when $G$ is any countable set, but for our purposes it is sufficient to consider only groups.} Throughout this section we follow the definitions of the space $\Omega$, the shift $\theta$, and the coordinate functions $X_n(w), n \in  \Z_+$, stated in the Subsection \ref{subsection: random walks}.

In the following two subsections we recall the main probabilistic results from \cite{MathieuSisto2020}, together with their variants that we will use in the proof of Theorem \ref{CLT acylindrically hyperbolic}.
\subsection{Defective adapted cocycles and the central limit theorem}
A sequence $\mathcal{Q}=\{Q_n\}_{n\ge 1}$ of maps $Q_n:\Omega\to \R$ that are $\sigma(X_1,\ldots, X_n)$, for each $n\ge 1$, is called a \emph{defective adapted cocycle}, and which we abbreviate as \textit{DAC}. We will use the convention $Q_0\equiv 0.$ 

In this paper we consider defective adapted cocycles constructed from random walks on groups. Whenever $Z_n, n \geq 0$, is a random walk on a countable group $G$, and $f:G \to \mathbb{R}$ is a function, there is an associated DAC defined by $Q_n(w)=f(Z_n(w))$ for each $ n \geq 1$. In particular, if $d$ is a word metric on $G$, then we call the DAC associated with function $f(g)=d(\id_G,g)$ the \emph{length defective adapted cocycle on $G$}.

The \emph{defect} of $\mathcal{Q}$ is the collection of maps $\Psi=\{\Psi_{m,n}\}_{m,n\ge 0}$ defined by
\[
\Psi_{m,n}(w)=Q_m(w)+(Q_n\circ \theta^m)(w) - Q_{n+m}(w), \text{ for each }w\in \Omega \text{ and }m,n\ge 0.
\] 

In particular, the defect of the length DAC is given by
\[\Psi_{m,n} (w)=d(\id_G,  Z_m)+d(Z_m,Z_{m+n}) -d(\id_G, Z_{m+n})\text{ for each }w\in \Omega \text{ and }m,n\ge 0.\]

\begin{defin}
	For $p \geq 1$, we say that $\mathcal{Q}$ has \emph{finite $p$-th moment}, if $\E^{\mu}[|Q_1|^{p}] < \infty$. Moreover, if $\sup_{m,n\ge 0} \left\{ \E^{\mu}\left[|\Psi_{m,n}|^p\right]\right\}  <\infty$ then we say that $\mathcal{Q}$ satisfies the \emph{$p$-th moment deviation inequality}.   
\end{defin}

The following theorem, essentially a combination of \cite[Theorem 4.2]{MathieuSisto2020} and \cite[Theorem 3.3]{MathieuSisto2020}, shows that the central limit theorem holds for defective adapted cocycles that satisfy the second moment deviation inequality.

\begin{thm}[{\cite[Theorem 4.2]{MathieuSisto2020}}] \label{thm: MS general CLT with constant deviation ineq}
	Let $G$ be a countable group endowed with a probability measure $\mu$. Consider $\mathcal{Q}$ a defective adapted cocycle on $\Omega=G^{\Z_{+}}$, and denote by $\{\Psi_{m,n}\}_{m,n\ge 0}$ its defect. Suppose that $\mathcal{Q}$ satisfies $\E^{\mu}[|Q_1|^{2}]<\infty$ and	$$\sup_{m,n\ge 0} \left\{ \E^{\mu}\left[|\Psi_{m,n}|^2\right]\right\}  <\infty$$
	
	Then, there exist real constants $\ell =\ell(\mu,\mathcal{Q})$ and $\sigma=\sigma(\mu, \mathcal{Q}) \in \R$ such that the random variables $\frac{1}{\sqrt{n}}\left(Q_n-\ell n\right)$ converge in law to a centered Gaussian random variable with variance $\sigma^2$.

	Moreover, the constant $\ell$ is also equal to the limit of $\frac{Q_n}{n}$ in $L^1(\Omega, \mu^{\Z_+})$ as $n\to \infty$.\footnote{The statement in the last sentence is proved in \cite[Theorem 3.3]{MathieuSisto2020}.}
\end{thm}

Furthermore, if we assume that $ \mathcal{Q}$ has finite $p$-th moment and satisfies $p$-th moment deviation inequality, the following theorem, see \cite[Theorem 4.9]{MathieuSisto2020}, establishes an upper bound on $p$-th moment of $Q_n - \E^{\mu}[Q_n]$.

\begin{thm}[{\cite[Theorem 4.9]{MathieuSisto2020}}] \label{thm: MS moments in the constant case}
	
	Let $\mathcal{Q}$ be a defective adapted cocycle that has a finite
	$p$-th moment and satisfies the $p$-th moment deviation inequality with respect to the
	probability measure $\mu$. Then, for any $p>1$, there exists a constant $C=C(\mu, p,\mathcal{Q})>0$ such that $\E^{\mu} | Q_n - \E^{\mu}[Q_n]|^p \leq C n^{p/2}$ for all $n\ge 1$.
\end{thm}

\subsection{Deviation inequalities with slowly growing upper bound}

In this subsection, we explain the generalization of the results obtained by Mathieu and Sisto in \cite{MathieuSisto2020} to the case where the moments of the defect $\Psi_{m,n}$ do not admit a constant bound, but grow slowly with respect to $m+n$. We believe that these facts are to large extent known to many experts in the field, see for example the proof of \cite[Theorem A]{ChawlaCLT}. However, since we did not find the precise and general statements in the literature, and we believe that they might be useful in the future works, we will state them and explain required technical modifications in the original arguments from \cite{MathieuSisto2020} below.

We start with the generalization of Theorem 4.2 from \cite{MathieuSisto2020}.
We say that the function $g: \mathbb{N} \rightarrow \mathbb{R}_{+}$ satisfies the \emph{slow growth assumption} (\textit{SG}) if it is non-decreasing and there exist a constant $K >0$ and $\epsilon >0$ such that

$$g(n) \leq K \frac{\sqrt{n}}{(\log n)^{1+ \epsilon}} \text{ for all }n\ge 2. $$

The upper bound in the above inequality is chosen for technical reasons. This is an almost optimal function that allows \cite[Theorem 2]{Hammerslet1962} to be applied to the function $h(n)=C_1\sqrt{n}g(n)$ in Item (2) in the proof of Theorem \ref{thm: MS Slow growth CLT} below.

\begin{thm} \label{thm: MS Slow growth CLT}
	Suppose  that $\mathcal{Q}$ is a non-negative DAC with a finite second moment and such that the defect 	\[
	\Psi_{m,n} := Q_{m+n} - (Q_m +  Q_{n}(\theta^m \omega))
	\]
	satisfies the inequality
	\[
	\E^{\mu}[|\Psi_{m,n}|^2] \le g(n+m) 
	\]
	for all $m,n \geq 1$, for some function $g$ that satisfies the slow growth assumption (SG).
	Then the CLT in the sense of \cite[Theorem 4.2]{MathieuSisto2020} holds for $\mathcal{Q}$.	
\end{thm}

\begin{proof}
	The general strategy in the proof of \cite[Theorem 4.2]{MathieuSisto2020}  works in our case as well, however, one needs to do the following technical modifications.

	\begin{enumerate}
		\item Instead of the linear upper bound on the variance of $ Q_n$ obtained in \cite[Theorem 4.4]{MathieuSisto2020} one can first derive the following weaker inequality that holds for every $n \geq 2$
		\begin{align}\label{var_weak}
			\mathbb{V}^{\mu}(Q_n) \leq C n g(n)
		\end{align}
		for some constant $C$ that depends only on $\mu$, $g$ and $\E^{\mu}|Q_1|^2$. The only change relative to the original argument in \cite[Theorem 4.4]{MathieuSisto2020} is to replace $\tau_2$ with $g(n)$. 
		\item Next, we describe the  modifications in the proof of Theorem 4.1 in \cite{MathieuSisto2020}. The assumption (SG) guarantees that the function $h (n)=C_1  \sqrt{n}  g(n)$ satisfies the conditions in \cite[Theorem 2]{Hammerslet1962} for any constant $C_1 >0$. Hence,  the conclusion of Lemma 4.5 still holds if we assume that the sequence $a_n, n \geq 1$ satisfies the inequality $$a_{m+n} \leq a_m+a_n+C_2 \sqrt{n+m}  g(n+m)$$
		for some large enough constant $C_2$. Notice that the inequality
		
		\begin{align*}
			|\mathbb{V}^{\mu}[Q_{n+m}] -\mathbb{V}^{\mu}[Q_{m}] -\mathbb{V}^{\mu}[Q_{n}]|  \\ \leq  \mathbb{V}^{\mu}[\Psi_{m,n}]+2 \sqrt{\mathbb{V}^{\mu}[Q_{m}] +\mathbb{V}^{\mu}[Q_{n}]}\sqrt{\mathbb{V}^{\mu}[\Psi_{m,n}]}
		\end{align*} 
		obtained in the proof of Theorem 4.1 is still valid in our setting. Moreover, since $g$ is non-decreasing, Inequality \eqref{var_weak} implies that $\mathbb{V}^{\mu}[Q_{m}] +\mathbb{V}^{\mu}[Q_{n}] \leq 2 C(n+m) g(n+m)$. 
		
		Since we also have the inequalities $\E^{\mu} |\Psi_{m,n}|^2 \leq g(n+m)$ and $ \mathbb{V}^{\mu}[\Psi_{m,n}]\leq \E^{\mu} |\Psi_{m,n}|^2$, it follows that for a sufficiently large constant $C_2>0$, the inequality
		
		$$\mathbb{V}^{\mu}[\Psi_{m,n}]+2 \sqrt{\mathbb{V}^{\mu}[Q_{m}] +\mathbb{V}^{\mu}[Q_{n}]}\sqrt{\mathbb{V}^{\mu}[\Psi_{m,n}]} \leq C_2 \sqrt{n+m} g(n+m) $$
		
		holds for all $m,n \geq 1$. It remains to apply Lemma 4.5 and conclude that the limit $ \sigma^2 (\mu; \mathcal{Q})=\lim_{ n \rightarrow \infty} \frac{\mathbb{V}^{\mu}[Q_{n}]}{n} $
		exists and is finite.
		
		\item The statement of Lemma 4.6 from \cite{MathieuSisto2020} remains correct, but the estimates in the proof need to be modified as follows.

		The inequality (4.10) on page 977 is replaced with $$ \mathbb{V}^{\mu}[\sum_{j=0}^m \gamma_j] \leq (\log n)^2 g(n). $$ 
		
		The inequality (4.11) on page 977 is replaced with
		$$\mathbb{V}^{\mu}[R_M] \leq 2^{2M+1}(\chi_2+M^2 g(n)) .$$
		
		The inequality (4.12) on page 978 is replaced with 
		\begin{align*}
			\mathbb{V}^{\mu}\Big[\sum_{i=M+1} ^m \sum_{j=0}^{\lceil n2^{-i}\rceil-1} \Psi_{2^{i-1}, 2^{i-1}} \circ \theta^{j2^{i}}\Big] \leq n \Big(\sum_{i=M+1}^{m} \sqrt{\frac {g(2^i)} {2^{i}}} \Big)^2 \\ \leq 
			n\Big(\sum_{i=M+1}^{\infty} \sqrt{\frac {g(2^i)} {2^{i}}} \Big)^2.
		\end{align*}
		
		Notice that the slow growth assumption (SG) implies that the series $$ \sum_{i=M+1}^{\infty} \sqrt{\frac {g(2^i)} {2^{i}}} $$
		converges.
		As a result, one obtains the inequality
		\begin{align*}
			\frac{1}{n}\mathbb{V}^{\mu}\Big[Q_n - \sum_{j=0}^{\lceil n 2^{-M}\rceil-1} Q_{2^M}\circ \theta^{j2^M} \Big] \leq
			\\ \frac{1}{n}\Big( \log n \sqrt{g(n)}+ 2^{M+1}\sqrt{\chi_2+M^2 g(n)}+ \sqrt{n}\Big(\sum_{i=M+1}^{\infty} \sqrt{\frac {g(2^i)} {2^{i}}} \Big) \Big)^2.
		\end{align*}
		Afterwards, the proof is completed as in \cite{MathieuSisto2020}. Notice that the slow growth assumption (SG) implies the inequality
		$$ \frac{(\log n) \sqrt{g(n)}}{\sqrt{n}} \leq \sqrt{K n^{-\frac{1}{2}} (\log n)^{1- \epsilon}}, $$
		
		so when $ n \rightarrow \infty$,  we have $$\lim_{n \rightarrow \infty } \frac{(\log n) \sqrt{g(n)}}{\sqrt{n}} =0.$$ 
		
		Similarly, $$ \lim_{n \rightarrow \infty } \frac{2^{M+1}\sqrt{\chi_2+M^2 g(n)}} {\sqrt{n}} =0.$$
		
		Finally, one takes $M \rightarrow  \infty$ and uses convergence of the series  $\sum_{i=1}^{\infty} \sqrt{\frac {g(2^i)} {2^{i}}} $  to complete the proof of Lemma 4.6.
		
		\item The remainder of the proof can be completed as in \cite{MathieuSisto2020}, but with a different argument showing that one can replace $\E^{\mu}[Q_n]$ with $l(\mu, \mathcal{Q})n$, where $l(\mu, \mathcal{Q})$ is the rate of escape of $\mathcal{Q}$. 
		Instead of using Theorem 3.4 in \cite{MathieuSisto2020},
		we are going to prove that, for all $n \geq 1,$ we have
		
		$$\left| \frac{1}{n}\E^{\mu}[Q_n] -l(\mu, \mathcal{Q}) \right| \leq C_3 n^{-3/4} $$
		
		for some sufficiently large constant $C_3$.
		
		Notice that the slow growth assumption (SG) guarantees that there exists a constant $C_4>0$ such that 
		$$\E^{\mu} |\Psi_{m,n}| \leq C_4 (m+n)^{1/4} \text{ for all }m,n\geq 1. $$
		
		Therefore, we may apply  \cite[Theorem 2]{Hammerslet1962} to $h(x)=C_4x^{1/4}$ and the sequence $a_n= \E^{\mu}Q_n$ to conclude that the asymptotic rate of escape $l(\mu, \mathcal{Q})$ is a non-negative real number. Moreover, the application of this theorem to the sequences $a_n= \E^{\mu}[Q_n]$ and $b_n= - \E^{\mu}[Q_n]$ yields 
		$$-\frac{1}{n}\E^{\mu}[Q_n] \geq - l(\mu, \mathcal{Q})+C_4n^{-3/4} -C_4 \frac{32}{3}\Big(2n \Big)^{-3/4} $$
		and 
		
		$$\frac{1}{n}\E^{\mu}[Q_n] \geq  l(\mu, \mathcal{Q})+C_4 n^{-3/4} -C_4 \frac{32}{3}\Big(2n\Big)^{-3/4}.
		$$
		Here we also used the trivial bound $$\sum_{x=2n}^{\infty}\frac{1}{x^{\frac{3}{4}}(x+1)} \leq 2 \int_{2n}^{\infty}x^{-\frac{7}{4}} dx= \frac{8}{3}\Big(2n\Big)^{-3/4}$$
		Therefore, we have $$| \frac{1}{n}\E^{\mu}[Q_n] -l(\mu, \mathcal{Q})| \leq C_3 n^{-3/4}, \text{ where } C_3=C_4\Big(\frac{32}{3}2^{-3/4}-1\Big).$$
	\end{enumerate}
	
\end{proof}

Theorem \ref{thm: MS moments in the constant case} can be restated as follows.

\begin{thm}\label{thm: moments_general case}
	Let $\mathcal{Q}$ be a defective adapted cocycle that has a finite
	$p$-th moment and  satisfies the inequality $$\E^{\mu}|\Psi_{m,n}|^p \leq g(m+n)$$ for all $m,n \geq 0$ for some non-decreasing function $g$. Then, for any $p>1$ there exists a constant $C$ that depends on $p$, $\mu$ and $\mathcal{Q}$, such that:
	
	$$\E^{\mu} | Q_n - \E^{\mu}{Q_n}|^p \leq C n^{p/2}g(n)$$
\end{thm}

\begin{proof}
	The original proof of \cite[Theorem 4.9]{MathieuSisto2020} still works if the inequality (4.13) on page 980 is replaced by  \[ \E^{\mu}\left[|Y_j|^p \right]\leq 6^{p-1}(2\chi(\mu; \mathcal{Q})+4g(n)).\]
\end{proof}
In Section \ref{section: final proof CLT} we will apply the results from this subsection to the length cocycle in the case when $G$ is a lamplighter group with an acylindrically hyperbolic base group $H$, $d$ is a standard metric on $G$, and $\mu$ has finite exponential moment and projects onto a non-elementary probability measure on $H$. 
In particular, we will show that in this case the moments of the defect are bounded by polynomials in $\log(m+n)$. 
More precisely, we will prove that for any $p \geq 1$, there exists a constant $K_p >0$ such that
\[
\E^{\mu}[(d(\id_G,Z_m)+d(\id_G,Z_m^{-1}Z_{m+n}) -d(\id_G,Z_{n+m}))^p]=\E^{\mu}[|\Psi_{m,n}|^p]\le K_p \log(n+m)^{5p} , 
\]
for all $n,m \ge 0.$

\section{Deterministic upper bounds on the defect of the TSP}\label{section: combinatorial}

The main goal of Sections \ref{section: combinatorial} and \ref{section: final proof CLT} is to establish the upper bounds on the moments of the defect of the length cocycle in the case of the acylindrically hyperbolic base group. Our argument will be divided into two parts. The first part is a geometric and combinatorial argument where we obtain deterministic bounds on the defect of the TSP. This is the main focus of Section \ref{section: combinatorial}. 
In the second part of the proof, covered in Section \ref{section: final proof CLT}, we show that, for a generic trajectory of the random walk, the deterministic bounds obtained in the first step are sufficiently sharp to ensure that the desired inequalities on the moments of the defect are true.

\begin{lem}\label{comb_lemma_1d}
	Let $A,B,C$ be three distinct points in the geodesic metric space $(X,d)$, and let $\gamma$ be a geodesic connecting $A$ and $C$. Let $R=d(A,B)$ and assume that $D>0 $ is a number such that $d(B, \gamma) \leq D$,  $R \geq 4D$, and $d(A,C) \geq R+ 4D$. Consider two finite subsets of $X$, $L_1$ and $L_2$, such that the following two conditions hold. 
	\begin{enumerate}
		\item The set $L_1 \cup L_2$ belongs to the $D$-neighborhood of $\gamma$.
		\item \label{separation}For any $x \in L_1$ we have $d(A,x) \leq R+ 4D$ and for any $y \in L_2$ we have $d(A,y) \geq R-4D $.
	\end{enumerate}
	Let $N$ denote the number of points $ x \in L_1 \cup  L_2$ such that $ R-4D  \leq d(A, x) \leq R+4D$. Then for any set $ L_3$ such that $L_1 \triangle L_2  \subseteq L_3 \subseteq L_1 \cup L_2$ we have
	
	$$ 0 \leq \mathrm{TSP}( A, L_1, B) +\mathrm{TSP} (B, L_2, C) - \mathrm{TSP} (A, L_3, C) \leq 24(N+1) D.$$
	
\end{lem}

\begin{figure}
	\centering
	\includegraphics[]{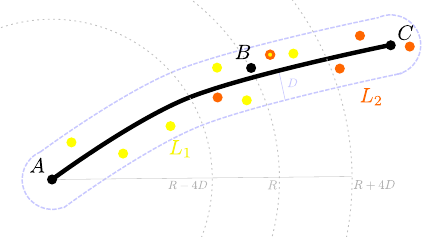}
	\caption{Hypotheses of Lemma \ref{comb_lemma_1d}.}
	\label{pic: Lemma 4.1}
\end{figure}
Let us briefly explain the setting of Lemma \ref{comb_lemma_1d}, illustrated by Figure
\ref{pic: Lemma 4.1}. The parameter $D$ controls the distances from  the points to $\gamma$ and the diameter of $L_1 \cap L_2$. The parameter $R$ controls the separation between $L_1 \cup \{A\}$ and $L_2 \cup \{C\}$ - most of the points in $L_1$ are to the left of $B$ and most of the points from $L_2$ are to the right of $B$.  Under these hypotheses, we will show that for a large enough $L_3$, we can restructure the solution of $\mathrm{TSP}(A, L_3, C)$, at a cost controlled by $N$, to  obtain a path from $A$ to $C$ that first visits all points in $L_1$, then goes through $B$ and visits every point in $L_2$.

\begin{proof}[Proof of Lemma~\ref{comb_lemma_1d}]
	Notice that for fixed endpoints $P$ and $Q$, the $\mathrm{TSP}(P,S,Q)$ is a non-decreasing function of the finite subset $S \subset X$ (with the subsets ordered by inclusion). In particular, we have $\mathrm{TSP}(A,L_3,C) \leq \mathrm{TSP}(A,L_1 \cup L_2, C)$. Then the first inequality follows from the fact that the concatenation of any two solutions of $\mathrm{TSP}( A, L_1, B)$ and $\mathrm{TSP}(B,L_2,C)$ at $B$ produces a path of length $\mathrm{TSP}( A, L_1, B) +\mathrm{TSP} (B, L_2, C)$ that starts at $A$, visits every point in $L_1 \cup L_2$ and ends at $C$, and hence has the length at least $\mathrm{TSP}(A,L_1 \cup L_2, C)$.
	
	Similarly, to prove the second inequality, it suffices to consider the case where $L_3= L_1 \Delta L_2$. The plan of the proof in this case is as follows. We will denote by $N_D(\gamma)$ the $D$-neighborhood of $\gamma$. Let us partition $N_D(\gamma)$ into three sets, the initial part $\mathcal{I}=N_D(\gamma) \cap B_{R-4D}(A)$, the middle part $\mathcal{M}=N_D(\gamma) \cap \{ x \in X |\, R-4D < d(A,x) < R+4D\}$ and the terminal part $\mathcal{T}=N_D(\gamma) \cap \{ x \in X |\, d(A,x) \geq R+4D\}$. Let $\alpha$ be any solution of $\mathrm{TSP}(A, L_3, C)$. We are going to show that after removing all geodesic segments of $\alpha$ that connect a node in $\mathcal{I}$ with a node in $\mathcal{T}$ or have at least one node in $\mathcal{M}$, the remaining subpaths of $\alpha$ could be combined with a sufficiently small family of new geodesic segments to obtain a path $\beta$ with the following properties. The path $\beta$ starts at $A$, visits every point in $L_1$, then goes through $B$, visits every point in $L_2$ afterwards, and finishes at $C$. By the definition of $\mathrm{TSP}$, the length of $\beta$ is at least $\mathrm{TSP}( A, L_1, B) +\mathrm{TSP} (B, L_2, C)$, and therefore, we only need to ensure that the total length of the new segments that we added does not exceed the total length of the segments that we removed plus $24(N+1)D$.
	
	We start with two elementary geometric inequalities which we will use to control the length of $\beta$. Let $B_1$ and $B_2$ be the unique points on the geodesic $\gamma$ such that $d(A,B_1)=R-4D$, and $d(A, B_2) =R+4D$. 
	
	\begin{claim} \label{leap_projection}
		Let $P$ be a point in $\mathcal{I}$ and $Q$ be a point in $\mathcal{T}$.  Then the following inequality holds: $$|PB_1|+|B_2Q| \leq |PQ|.$$
	\end{claim}
	
	\begin{figure}
		\centering
		\includegraphics[]{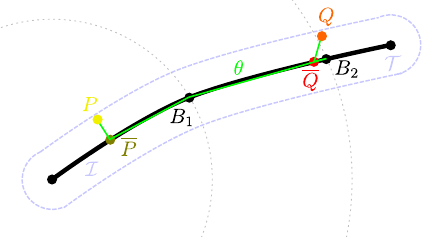}
		\caption{Illustration of the proof of Claim \ref{leap_projection}.}
	\end{figure}
	
	\begin{proof}
		
		Fix some projections $\overline{P} \in \pi_{\gamma}(P)$  and $\overline{Q} \in \pi_{\gamma}(Q)$ of $P$ and $Q$ on $\gamma$. 
		We claim that the path $\theta$ composed of the geodesic segments $P\overline{P}$, $\overline{P}B_1$, $B_1B_2$, $B_2 \overline{Q}$, and $\overline{Q}Q$ satisfies the  inequality:
		$$|\theta|=|P\overline{P}|+|\overline{P}B_1| +|B_1B_2|+|B_2 \overline{Q}|+|\overline{Q}Q| \leq |PQ|+8D.$$
		
		Since $P$ and $Q$ belong to $N_D(\gamma)$, both segments $P \overline{P}$ and $\overline{Q}Q$  have length at most $D$. Then, the triangle inequality implies that \begin{align} \label{ineq_1}
			|P\overline{P}|+|\overline{P}\overline{Q}| +|\overline{Q}Q| \leq |PQ|+ 2|P\overline{P}|+2|\overline{Q}Q| \leq |PQ|+4D   
		\end{align}
		
		Notice that if $d(A,\overline{P}) > d(A,B_1)= R-4D$, then $|\overline{P}B_1|\leq D$, and if $d(A,\overline{Q}) < d(A,B_2)= R+4D$, then $|B_2\overline{Q}|\leq D$. Therefore, whenever one of the segments $\overline{P}B_1 $ and $B_2\overline{Q}$ is not a subsegment of $\overline{P}\overline{Q}$, its traversals contribute at most $2D$ to the expression $ |\overline{P}B_1| +|B_1B_2|+|B_2 \overline{Q}| -|\overline{P}\overline{Q}| $, so  we have
		$$ |\overline{P}B_1| +|B_1B_2|+|B_2 \overline{Q}| \leq |\overline{P}\overline{Q}|+4D.$$
		
		Combining this with Inequality \eqref{ineq_1}
		we obtain the desired inequality:  $$|\theta| \leq |PQ|+8D.$$
		
		To prove the inequality $|PB_1|+|B_2Q| \leq |PQ|$, observe that the triangle inequality and our upper bound on $|\theta|$ imply the following inequalities:
		$$ |PB_1|+|B_2Q| \leq |P\overline{P}|+|\overline{P}B_1| +|B_2 \overline{Q}|+|\overline{Q}Q| \leq |PQ|+8D -|B_1B_2| =|PQ|$$
	\end{proof}

	\begin{claim}\label{step_projection}
		Let $P$ be a point in $\mathcal{I}$ and $Q$ be a point in $\mathcal{M}$. 
		Then we have $$|PB_1| \leq |PQ|+6D.$$ 
	\end{claim}
	\begin{proof}

		The proof is similar to the proof of  Claim \ref{leap_projection}. 
		We fix some projections $\overline{P} \in \pi_{\gamma}(P)$  and $\overline{Q} \in \pi_{\gamma}(Q)$ of $P$ and $Q$ on $\gamma$. Then we show that the path $\theta$ composed of the geodesic segments $P\overline{P}$, $\overline{P}B_1$, $B_1 \overline{Q}$, and $\overline{Q}Q$ satisfies the inequality:
		
		$$|\theta| =|P\overline{P}| +|\overline{P}B_1|+|B_1 \overline{Q}|+|\overline{Q}Q| \leq |PQ|+6D.$$
		
		While showing that $ |\overline{P}B_1|+|B_1 \overline{Q}| \leq |\overline{P}\overline{Q}| +2D$, we have two additional cases where the points $B_1, \overline{P},$  and $\overline{Q}$ appear along  $\gamma$ in the order $B_1, \overline{Q}, \overline{P}$ or in the order $\overline{Q}, \overline{P}, B_1$. In the first case we have $|B_1\overline{Q}|\leq | \overline{P}B_1| \leq D$ , and in the second, we have $| \overline{P}B_1| \leq |B_1\overline{Q}| \leq D $. In any case, the inequality $ |\overline{P}B_1|+|B_1 \overline{Q}| \leq |\overline{P}\overline{Q}| +2D$ holds.
		
	\end{proof}
	Of course, the inequalities in Claim \ref{step_projection} remain true if we assume that $P$ is in $\mathcal{T}$ instead of $\mathcal{I}$, and $B_1$ is replaced with $B_2$.

	For our construction of $\beta$, it will be convenient to classify parts of $\alpha$ as follows. We call \textit{the trace of $\alpha$  in $\mathcal{I}$} the union of all geodesic segments of $\alpha$ that connect two nodes from $\mathcal{I}$, and we will denote it by $\alpha_I$. It is easy to see that $\alpha_I$ is the union of maximal subpaths of $\alpha$ that have both endpoints in the nodes of $\alpha$ contained in $\mathcal{I}$. Let us list these paths in the order in which they are covered by $\alpha$: $p_1, p_2, \ldots, p_t$.  The trace of $\alpha$ in $\mathcal{T}$, denoted by $\alpha_T$, and the ordered list of its maximal subpaths $q_1, ..., q_r$ are defined similarly.  Furthermore, we will also split the geodesic segments of $\alpha$ that do not belong $\alpha_I$ or $\alpha_T$ and are not joining two nodes in $\mathcal{M}$ into two categories.
	We will call a geodesic segment $s_i$ \textit{a step}, if it connects a node in $\mathcal{M}$ with a node in $\mathcal{I} \cup \mathcal{T}$, and we will call $s_i$ \textit{a leap} if one of its nodes is in $\mathcal{I}$ and the other is in $\mathcal{T}$.  To construct the path $\beta$ we will first connect the subpaths of $\alpha_I$ and $\alpha_T$ by geodesic segments going through $B_1$ and $B_2$ respectively. 
	The precise procedure for the trace $\beta_I$ is described below.
	
	\begin{enumerate}
		\item Both $\alpha$ and $\beta$ start by traversing $p_1$.
		\item If $\beta$  is traversing $p_i$, and $p_i$ ends with the node $V_i$ while the path $p_{i+1}$ starts with the node $W_{i+1}$, $\beta$ will transition from $V_i$ to $W_{i+1}$ via the segments $V_iB_1$ and $B_1W_{i+1}$, and then track $p_{i+1}$ following $\alpha$.
		\item Once $\beta$ reaches the endpoint $V_t$ of the last path $p_t$, it proceeds to $B_1$ via the segment $V_tB_1$.
	\end{enumerate}
	
	The trace $\beta_T$ is constructed in a similar way, but this time the starting point of $q_1$ is connected to $B_2$ by a geodesic segment (the final node of $q_r$ is $C$ and it is also the end node of $\beta$).
	
	At this stage, we added to $\alpha_I$ and $\alpha_T$ the segments connecting all of the endpoints of $p_i , i=1, \ldots ,t$, except for $A$, with $B_1$, and the segments connecting all of the endpoints of $q_j, j =1, \ldots, r$, except for $C$, with $B_2$.  Notice that if one of the maximal subpaths consists of a single point $P$, then we are adding the corresponding segment $PB_i$ twice.
	
	Notice that each of the paths $p_2, \ldots, p_t, q_1, \ldots, q_r$ is immediately preceded in $\alpha$ by a uniquely determined step or leap (otherwise, these would not be the maximal subpaths of $\alpha_I$ and $\alpha_T$). Moreover, each of the paths $p_1, \ldots, p_t, q_1, \ldots, q_{r-1}$ is immediately followed by a uniquely determined step or leap of $\alpha$. 
	Therefore, any node of $\alpha$, except for $A$ and $C$, is listed as an endpoint of some of the paths $p_1, \ldots, p_t, q_1, \ldots , q_r $  as many times, as it is listed as endpoint of steps and leaps of $\alpha$. Here we also use the convention that if some of the maximal paths consists of a single node, then this point is counted twice - once as a starting point and once as an ending point of the single node maximal path.
	
	Now we are ready to prove the upper bound on the total length of the new segments.
	Pair each leap $PQ,  P \in \mathcal{I},  Q \in \mathcal{T}$ with the segments $PB_1$ and $B_2Q$ that were added when we constructed $\beta_I$ and $\beta_T$.  Notice that Claim \ref{leap_projection} implies that $|PB_1|+|B_2Q| \leq |PQ|$.  It is also easy to see that each traversal of an added segment can be associated with at most one leap. Therefore, the total length of the new segments paired with the leaps does not exceed the total length of all of the leaps. The remaining added segments are in a natural one-to-one correspondence with the steps of $\alpha$ - each step $PQ, P \in \mathcal{I} \cup \mathcal{T}, Q \in \mathcal{M}$ is associated with the added segment $PB_1$ if $P \in \mathcal{I}$ or $PB_2$ if $P \in \mathcal{T}$. Notice that in this case, Claim \ref{step_projection} implies that $|PB_i| \leq |PQ|+6D$. Moreover, each node in $\mathcal{M}$ is the endpoint of at most two steps, so the number of steps is bounded by $2N$. As a result, the total increase in length from the segments associated with steps is at most $12ND$.
	
	Since there was no increase coming from the segments paired with leaps, we conclude that $|\beta_I|+|\beta_T|$ is at most $|\alpha|+12ND$. Moreover, $\beta_I$ already covers all the nodes in  $L_1 \cap \mathcal{I}$ and $\beta_T$ covers all the nodes in $L_2 \cap \mathcal{T} $, because $L_1 \Delta L_2  \cap \mathcal{I}=L_1 \cap \mathcal{I}$ and $L_1 \Delta L_2  \cap \mathcal{T}=L_2 \cap \mathcal{T} $ by Condition \eqref{separation} in the statement of this lemma. Furthermore, by construction, $\beta_I$  ends at $B_1$, and $\beta_T$ starts at $B_2$, and hence, to complete the construction of $\beta$, we only need to pick an arbitrary path that starts at $B_1$, then visits every point in $L_1 \cap \mathcal{M}$, goes through $B$, then visits every point in $L_2 \cap \mathcal{M}$, and ends at $B_2$. Since the diameter of $\mathcal{M}$ does not exceed $12D$, and the path should have at most $N+3$ nodes, the length of such a path may be bounded above by $12D(N+2)$. In total, we have the desired upper bound $$|\beta| \leq |\alpha|+12ND+12D(N+2) =|\alpha|+24(N+1)D$$
	This completes the proof of the lemma.
	
\end{proof}

\section{Bounds on the moments of the defect and the CLT}\label{section: final proof CLT}

In this section we apply Lemma \ref{comb_lemma_1d} to obtain upper bounds on the moments of the defective adapted cocycles when the base group $H$ is acylindrically hyperbolic. 

\begin{thm} \label{thm: moments of the defect}
	Let $G=A \wr H$ be a lamplighter group with a finitely generated acylindrically hyperbolic base group $H$. Let $d$ be a standard metric on $G$. Assume that the probability measure $\mu$ on  $G$ has finite exponential moment, and its projection onto $H$, $\mu_H$ is a non-elementary measure. Let $\mathcal{Q}$ be the corresponding length defective adapted cocycle and $\Psi_{m,n}, \,m, n  \geq 1,$ be its defects. Then, for every $p \geq 1$, there exists a constant $K_p(\mu,d)$ such that for any $m,n \in \mathbb{N}$ the following inequality holds:
	\[
	\E^{\mu}|\Psi_{m,n}|^p\leq K_p \log(n+m)^{5p}.
	\]
	
	Moreover, if we assume that $\mu$ has finite support, then the right-hand side of the inequality could be strengthened to
	\[
	\E^{\mu}|\Psi_{m,n}|^p\le K_p \log(n+m)^{2p}.
	\]
\end{thm}

We will use the following two propositions in our proof. These results will guarantee logarithmic growth of the deviation from the geodesic and linear progress of the trajectory of the random walk projected to $H$. We note that these propositions are the only parts of our arguments that require $H$ to be an acylindrical hyperbolic group.
The first proposition is a straightforward corollary of Theorem 9.1 and Theorem 10.7 in \cite{MathieuSisto2020}.

\begin{prop}\label{prop: tracking}
	Let $H$ be a finitely generated acylindrically hyperbolic group and let $\mu_H$ be a symmetric non-elementary probability measure on $H$ with finite exponential moment.  Choose an arbitrary finite symmetric generating set of $H$ and let $d_H$ be the corresponding word metric on $H$.  Finally, we denote by $\overline{Z}_n$ the position of the random walk driven by $\mu_H$ at time $n$. Then the following statements hold.

	\begin{enumerate}
		\item There exists a constant $K$ such that  for any $n \geq 1$
		\begin{align*}
			\mathbb{P}^{\mu_H} \left( d_H(\overline{Z}_n,e_H) \leq n/K \right) \leq Ke^{-n/K}
		\end{align*}
		\item  For any $p \ge 1$, there is a constant $C_1$ such that for any $n \geq 1$ and any geodesic segment  $\alpha$ joining $\id$ and $Z_n$ we have
		\begin{align*}
			\mathbb{P}^{\mu_H} \left( \max_{0 \leq k \leq n} d_H(\overline{Z}_k,\alpha) \geq C_1 \log n \right) \leq C_1/n^{2p}
		\end{align*}
	\end{enumerate}
\end{prop}

\begin{proof}
	The first statement immediately follows from Theorem 9.1, Remark 10.2 and statement 1 in Proposition 10.3 in \cite{MathieuSisto2020}. The second statement follows from Theorem 10.7 and Remark 10.2 in \cite{MathieuSisto2020}.
\end{proof}

We note that if $\mu_H$-random walk $\overline{Z}_{k}$ on $H$ in Proposition \ref{prop: tracking} is a projection of a $\mu$-random walk on $G=A \wr H$, then, since $\mathbb{P}^{\mu_H}$ is the pushforward of $\mathbb{P}^\mu$, the inequalities in Proposition  \ref{prop: tracking} hold as inequalities with respect to $ \mathbb{P}^{\mu}$. We will keep this convention for the remainder of the article.
\begin{prop}
	
	\label{lem: uniform progress}
	Under the same hypotheses as Proposition \ref{prop: tracking},
	for any $p \ge 1$, there is a constant $K_0 \geq 1 $ such that with for 
	any $n,m$, with probability at least $1-\frac{1}{(n+m)^{2p}}$ the following holds. For all
	$i,j\in \{0,\ldots n+m\}$ such that
	$|i-j|\ge K_0\log(n+m)$, we have $d_H(\overline{Z}_i,\overline{Z}_j)\ge \frac{|i-j|}{K_0}$.
\end{prop}
\begin{proof}
	This follows from Proposition \ref{prop: tracking} together with a union bound.
\end{proof}

\begin{proof}[Proof of Theorem \ref{thm: moments of the defect}]
	
	We will divide the proof into three steps. In Step $1$ we will establish some a priori bounds on the moments of the defect $\Psi_{m,n}$ and deal with the case where either $m$ or $n$ is sufficiently small relative to their sum. In the second step, we will prove the desired upper bound for the general case of the measure with finite exponential moment under the assumption that $m$ and $n$ are sufficiently large. Finally, in Step $3$ we will briefly explain how to modify the argument from Step $2$ to decrease the right hand side if we assume that $\mu$ has finite support.
	
	For $p \geq 1$ denote by $\kappa_p(\mu,d)$ the $p$-th moment of measure $\mu$ with respect to distance $d$:
	
	$$ \kappa_p(\mu,d) =\E^{\mu}d(\id,Z_1)^p.$$

	\subsection*{Step 1}
	Notice that we have the following elementary upper bounds:
	
	\begin{align*}
		|\Psi_{m,n}|=|d(\id,Z_m)+d(Z_m, Z_{m+n}) -d(\id, Z_{m+n})| \\ \leq d(\id,Z_m)+|d(Z_m, Z_{m+n}) -d(\id, Z_{m+n})| \leq 2d(\id,Z_m)    
	\end{align*}
	and 
	\begin{align*}
		|\Psi_{m,n}| =|d(\id,Z_m)+d(Z_m, Z_{m+n}) -d(\id, Z_{m+n})| \\
		\leq d(Z_m,Z_{m+n})+|d(\id, Z_{m}) -d(\id, Z_{m+n})|
		\leq 2d(Z_m,Z_{m+n})   
	\end{align*}
	
	Therefore, we can derive the following upper bound for any $p \geq 1$: 
	\begin{align*}
		|\Psi_{m,n}|^p \leq  2^pd(\id,Z_m)^p \leq    2^p \Big(\sum_{k=0}^{m-1}d(Z_k, Z_{k+1}) \Big)^p 
		\leq 2^p m^{p-1}\sum_{k=0}^{m-1}d(Z_k, Z_{k+1})^p.
	\end{align*}
	
	Since $d(Z_k, Z_{k+1})$ are i.i.d. with the same distribution as $d(Z_0, Z_1)=d(\id,Z_1)$, taking the expectation, we get 
	
	$$\E^{\mu}|\Psi_{m,n}|^p \leq 2^pm^p \kappa_p(\mu,d).$$
	
	The inequality $\E^{\mu}|\Psi_{m,n}|^p \leq 2^pn^p \kappa_p(\mu,d)$ can be proved in the same way.

	Therefore, if $ \min {(m,n)} \leq M \log (m+n)^2$ for some fixed $M>0 $,  we get the inequality 
	\[\E^{\mu}|\Psi_{m,n}|^p \leq 2^pM^p \log{(m+n)}^{2p} \kappa_p(\mu,d),\]
	which shows that under the assumption $ \min {(m,n)} \leq M \log (m+n)^2$ the moment $\E^{\mu}|\Psi_{m,n}|^p$ grows no faster than a polynomial of degree $2p$ in $\log(m+n)$.
	
	\subsection*{Step 2}
	From now on, we will assume that $ \min {(m,n)} \geq M \log (m+n)^2$ for some large constant $M$ that depends only on $\mu$ and $d$ and will be defined precisely later.
	
	Since $|\Psi_{m,n}|$ has all moments,  we can limit the expectation of $|\Psi_{m,n}|^p$ over the event $\{|\Psi_{m,n}| \geq  (n+m)^2 \}$ as follows:
	\begin{align*}
		\E^{\mu}\big[|\Psi_{m,n}|^p: |\Psi_{m,n}| \geq  (n+m)^2 \big] \leq  \frac{\E^{\mu}|\Psi_{m,n}|^{2p} }{ (n+m)^{2p}} \leq \frac{2^{2p}(m+n)^{2p} \kappa_{2p}(\mu,d)}{ (n+m)^{2p}}=2^{2p}\kappa_{2p}(\mu,d).   
	\end{align*}
	Hence, the expectation over this set has a constant upper bound that does not depend on $m,n$.
	
	Therefore, in our subsequent estimates it suffices to bound the expectation of $|\Psi_{m,n}|$ on the event $S_{m,n} =\{|\Psi_{m,n}| \leq  (n+m)^2 \}$.

	Let us fix some $p \geq 1$. Notice that for any event $A$ with $\mathbb{P}^\mu(A)  \leq \frac{M_1}{(m+n)^{2p}}$ for some constant $M_1 >0$, we have $\E^{\mu}\big[|\Psi_{m,n}|^p: A \cap S_{m,n} \big] \leq M_1 $, so it suffices to bound the expectation $\E^{\mu}\big[|\Psi_{m,n}|^2: A^C \cap S_{m,n} \big] $ taken over the intersection fo $S_{m,n}$ with the complement $A^c$.

	Our goal is to show that the trajectory of the random walk on $G$ up to time $m+n$ projected onto $H$ will satisfy the conditions of Lemma \ref{comb_lemma_1d} with $N$ and $D$ bounded above by polynomials of $\log(m+n)$ of low degree with probability at least $1- \frac{M_2}{(m+n)^{2p}}$ for some constant $M_2 > 0$. To simplify the notation, we will say that an event $A \subseteq \Omega$ that depends on $m$ and $n$  and is measurable with respect to $\sigma(X_1, \ldots, X_{m+n})$ holds with \textit{overwhelming probability} if $\mathbb{P}^{\mu}(A) \geq 1 -\frac{M_3}{(m+n)^{2p}}$ for some constant $M_3 >0$ that does not depend on $m$ and $n$.

	We will show that, with overwhelming probability, the lamps switched on at any given time $0 \leq k \leq m+n$ are sufficiently close to the position of $\overline{Z}_{k}$ in the Cayley graph of $H$.
	Since measure $\mu$ has finite exponential moment with respect to $d$, there exists $\alpha >0 $ such that $$e(\mu,d) =\E^{\mu}\exp(\alpha d(\id,X_1)) < \infty$$ 
	
	Therefore, by the Chebyshev inequality, there exists a constant $C_0 \geq 1$ such that for any $m,n \in \mathbb{N}:$ 
	\[\mathbb{P}^{\mu}\left(d(\id_G, Z_1)>C_0 \log(m+n)\right) \leq \frac{1}{(m+n)^{2p+1}}.\]
	But this implies the following chain of inequalities 
	\begin{align*}
		\mathbb{P}^{\mu}\left( \max_{k=1, \ldots, m+n} d(\id_G, X_k) > C_0 \log(m+n)\right) \leq 1 - \left(1-\frac{1}{(m+n)^{2p+1}}\right)^{m+n} \\
		\leq 1 -\left(1-\frac{1}{(m+n)^{2p}}\right) =\frac{1}{(m+n)^{2p}} .  
	\end{align*} 
	
	Notice that if $g \in G$ has the representation $(f,h)$ and $d(\id_G,g) \leq C_0 \log (m+n)$, then $|\supp{f}| \leq C_0 \log (m+n)$, and the distance from $\id_H$ to any element in $\supp{f} \cup \{h\}$ does not exceed $C_0 \log (m+n)$. Therefore, with overwhelming probability, we may assume that for every $k = 1, \ldots, m+n$ the lamps switched at the time $k$ are within the distance $C_0 \log (m+n)$ from the position $\overline{Z}_{k-1}$ in the Cayley graph of $H$, and $d_H(\overline{Z}_{k-1},\overline{Z}_{k}) \leq C_0 \log (m+n)$.
	
	Now, we will apply Lemma \ref{comb_lemma_1d} to the points $A=\id_H$, $B = \overline{Z}_m$, and $C=\overline{Z}_{n+m}$, where $L_1$ and $L_2$ are the sets of lamps switched between times $1$ and $m$,  and $m+1$ and $m+n$ respectively. We have $R=d_H(\id_H,\overline{Z}_m)$, and we need to verify that, with sufficiently high probability, there exists a positive number $D$ that satisfies the inequalities in the statement of Lemma \ref{comb_lemma_1d}. Notice that by Proposition \ref{lem: uniform progress}, we may assume that $R=d_H(\overline{Z}_0, \overline{Z}_m) \geq \frac{m}{K_0}$ and $ d_H(\overline{Z}_0, \overline{Z}_{m+n}) \geq \frac{m+n}{K_0}$ as soon as $m \geq K_0 \log (m+n)$. Moreover, by Item (2) in Proposition \ref{prop: tracking} applied to the geodesic segment $\gamma$ connecting $A$ and $C$, we may assume that 
	$$\max_{0 \leq k \leq m+n} d_H(\overline{Z}_k,\gamma) \leq C_1 \log (m+n).$$ 
	Since we know that every lamp from $L_1 \cup L_2$ is within $C_0 \log(m+n)$ distance from one of the points $\overline{Z}_k, k=0, \ldots, m+n$, we conclude that every lamp in $L_1 \cup L_2$ is within $(C_0+C_1)\log (m+n)$ distance from $\gamma$. Thus, $L_1 \cup L_2 \cup \{ B\} \subseteq N_D(\gamma)$ as soon as $D \geq(C_0+C_1)\log(m+n)$. To shorten the notation, we set $D_1=(C_0+C_1)\log(m+n)$.
	
	Now, notice that if for some $ 0 \leq i < j \leq n+m $ we have $$|d_H(A, \overline{Z}_i) - d_H(A,\overline{Z}_j)| \leq C_0 \log (m+n), $$ then the assumption $\{\overline{Z}_i, \overline{Z}_j\} \subseteq N_{D_1} (\gamma)$ implies that $$d_H(\overline{Z}_i,\overline{Z}_j) \leq 4D_1 +C_0 \log(m+n).$$ However, we also know that $d_H(\overline{Z}_i,\overline{Z}_j) \geq \frac{j-i}{K_0} $ or $j-i \leq K_0 \log(m+n) $, and thus we get $$ i < j \leq i+ K_0(4D_1+C_0 \log (m+n)).$$
	
	Since we know that $d_H(\overline{Z}_k, \overline{Z}_{k+1}) \leq C_0 \log (m+n)$ for all $k=0, \ldots, n+m-1$, the previous estimates imply that for any $j \geq m+K_0(4D_1+C_0 \log (m+n))$ we must have $d_H(A, \overline{Z}_j) \geq R$, and for any $0  \leq i \leq m-K_0(4D_1+C_0 \log (m+n))$ we have $d_H(A, \overline{Z}_i) \leq R$.  Moreover, for any $m \leq j \leq m+K_0(4D_1+C_0 \log (m+n))$ we will also have $$d_H(A, \overline{Z}_j) \geq R - C_0 \log (m+n) K_0(4D_1+C_0 \log (m+n)).$$
	
	In particular, no lamp from $L_2$ could be at the distance less than or equal to $R-D_2$ from $A$, where $$D_2 = C_0 K_0 \log (m+n) (4D_1+C_0 \log (m+n)) +C_0 \log (m+n).$$
	
	Furthermore, no lamp from $L_1$ could be at the distance greater than $R+D_2$ from $A$, and since we chose $C_0, K_0 \geq 1$, we automatically have $C_0K_0 \geq 1$, and as the result,  $D_2/4 \geq D_1$. Hence, if we set $D=D_2/4$, we get $L_1 \cup L_2 \cup \{ B\} \subseteq N_D(\gamma)$.

	Finally, since $D_1=(C_0+C_1)\log(m+n)$, $D$ is a quadratic function of $\log(m+n)$, and we may choose a constant $M$ large enough such that $M \log(n+m)^2 > 6DK_0$. Then, once we assume that $\min(m,n) \geq M \log(n+m)^2$, we will also immediately get $4D \leq \frac{m}{K_0} \leq d_H(A, B) =R$. Moreover, by Proposition \ref{lem: uniform progress}  we have that $6D \leq \frac{n}{K_0} \leq  d_H(B,C)$, and, since $d_H(B, \gamma) \leq D$, we have $$d_H(A,B)+d_H(B,C) \leq d_H(A,C)+2D.$$
	Hence, the following inequalities are valid:
	$$ 6D \leq \frac{n}{K_0} \leq  d_H(B,C)\leq d_H(A,C) -d_H(A,B)+2D =d_H(A,C) -R+2D,$$ so $d_H(A,C) \geq R+4D$.
	
	Therefore, all conditions of Lemma \ref{comb_lemma_1d} are satisfied, and we may conclude that 
	\[\left|\mathrm{TSP}(\id_H, L_1, \overline{Z}_m) +\mathrm{TSP}(\overline{Z}_m, L_2, \overline{Z}_{m+n}) -\mathrm{TSP}(\id_H,L_3,\overline{Z}_{m+n})\right| \leq 24(N+1)D,\]
	where $L_3$ denotes the set of lamps active at time $m+n$, and $N$ denotes the number of lamps from $L_1 \cup L_2$ in the region $$\mathcal{M} =\{h \in H| R-4D \leq d_H(\id_H,h) \leq R+4D , \, d_H(h, \gamma) \leq D\}.$$
	
	Moreover, $ |L_1 \cap L_2| \leq|L_1|+|L_2| -|L_3| \leq 2|L_1 \cap L_2|$, and since $L_1 \cap L_2 \subseteq \mathcal{M}$ we have $|L_1|+|L_2| -|L_3| \leq 2N$.
	
	As a result, we have the following inequality: 
	
	$$ |\Psi_{m,n}| \leq 24(N+1)D+2N.$$
	
	Next, we will establish upper bounds on $D$ and $N$ in terms of $\log(m+n)$.
	We already know that $D$ can be bounded above by $C_2 \log(m+n)^2$ for some constant $C_2 \geq 1$ that may be explicitly expressed as a polynomial in $K_0,C_0,C_1$. Moreover, $\mathcal{M}$ has the diameter at most $12D$, so Proposition \ref{lem: uniform progress} implies that the number of positions $\overline{Z}_k$ within $C_0 \log(m+n)$-neighborhood of $\mathcal{M}$ is bounded above by $K_0 (12D+2C_0 \log(m+n))+1$. Since every lamp in  $\mathcal{M}$ is within $C_0 \log(m+n)$  distance from the corresponding position $\overline{Z}_k$, and each position corresponds to at most $C_0 \log(m+n)$  lamps, we conclude that $$ N \leq (K_0 (12D+2C_0 \log(m+n))+1)C_0 \log(m+n).$$
	
	Therefore, we see that $N$ is bounded above by a polynomial of $\log(m+n)$ of the degree at most $3$, and as a result, $24(N+1)D+2N$ may be bounded above by $C_2 \log(m+n)^5$ for some constant $C_2 \geq 1$ that may be explicitly expressed as a polynomial depending only on the constants $K_0, C_0, C_1$ defined above.
	
	As a result, in the case when $\min (m,n) \geq M \log(m+n)^2$, we obtain the inequality 
	
	\[
	\mathbb{E}|\Psi_{m,n}|^p\leq C_2^p \log(n+m)^{5p} +M_2,
	\]
	
	for some constant $M_2 > 0$ which accounts for the expectation over the events of small probability  described above (these are the events where the trajectory $\overline{Z}_0, \ldots, \overline{Z}_{m+n}$ behaves irregularly).
	
	Combining this with the results obtained in Step $1$, we see that there exists a constant $K_p > 0$ such that for all $m,n \in \mathbb{N}$:
	
	\[
	\mathbb{E}|\Psi_{m,n}|^p\leq K_p \log(n+m)^{5p}.
	\]
	
	\subsection*{Step 3.} If we assume that $\mu$ has finite support, then we can run the same argument as in Step $2$, but in this case we may assume that for some constant $C_0 \geq 1$, for every $k = 1, \ldots, m+n,$ the lamps switched at time $k$ are within the distance $C_0$ from the position $\overline{Z}_{k-1}$ in the Cayley graph of $H$ and there are at most $C_0$ such lamps. Moreover, we may also assume that $d_H(\overline{Z}_{k-1},\overline{Z}_{k}) \leq C_0$ for every $k = 1, \ldots, m+n$. Therefore, in our estimate of $D$, the upper bound on the distance between consecutive positions $\overline{Z}_k$ and $\overline{Z}_{k+1}$ will be constant, and the upper bound on $D$ will be linear in $\log(m+n)$.  Moreover, the upper bound on $N$ will be the product of a linear function of $\log(m+n)$ with the largest possible number of active lamps in the element from $\supp{\mu}$, so this time it will also be linear in $\log(m+n).$ 
	As a result, the upper bounds obtained in both previous steps would not exceed $K_2\log(m+n)^{2p}$ for a suitable constant $K_2 \geq 1$, which completes the proof of the inequality.
\end{proof}

Now we are ready to complete the proof of Theorem \ref{CLT acylindrically hyperbolic}.

\begin{thm}[see Theorem 1.1]

	Let $H$ be a finitely generated  acylindrically hyperbolic group and let $A$ be any nontrivial finite group. Assume that $G=A \wr H$ is equipped with a standard metric $d$, and $\mu$ is a symmetric probability measure on $G$ with finite exponential moment such that its projection onto $H$ is a non-elementary measure. Then there exists $\sigma>0$ such that $$\frac{d(\id_G,Z_n) -\ell(\mu,d)n}{\sqrt{n}}$$  converges to $\mathcal{N}(0, \sigma ^2)$ in distribution.

	Moreover, for each $p \geq 1$, there exists a constant $C$ that depends only on $p$ and $\mu$, such that for all $n>1$, $$\E^{\mu}|d(\id_G,Z_n) -\E^{\mu}d(\id_G,Z_n)|^p \leq Cn^{\frac{p}{2}}(\log n)^{5p}. $$
\end{thm}

\begin{proof}
	The CLT for the drift follows from Theorem \ref{thm: MS Slow growth CLT} and Theorem \ref{thm: moments of the defect} applied to $p=2$. The fact that $\sigma>0$ follows from \cite[Theorem 4.12]{MathieuSisto2020}. We note that this the only part of the proof that uses symmetry of $\mu$, and it can be replaced by weaker assumptions, see the statement of \cite[Theorem 4.12]{MathieuSisto2020}.
	
	Finally, the upper bounds on the moments of $|d(\id_G,Z_n) -\E^{\mu}d(\id_G,Z_n)| $ follow from Theorem \ref{thm: moments of the defect} and Theorem \ref{thm: moments_general case}.
\end{proof}

\begin{rem}
	The last claim in Theorem~\ref{thm: moments of the defect} implies that in the case when $\mu$ has finite support, the bounds on central moments of $d(\id_G,Z_n)$ can be improved to 
	$$\E^{\mu}|d(\id_G,Z_n) -\E^{\mu}d(\id_G,Z_n)|^p \leq Cn^{\frac{p}{2}}(\log n)^{2p}.$$
\end{rem}

\subsection{The case of a finitely generated group of lamps}\label{subsec: f.g. lamps} In this subsection we describe the modifications to the proof of Theorem \ref{thm: moments of the defect} that lead to the proof of the CLT in the case where $A$ is an arbitrary finitely generated group. 

Let $S_A$ be a finite symmetric generating set of $A$, and $S_H$ be a finite generating set of $H$.
In this case, the standard generating set of $G=A \wr H$ is given by
\[ S \coloneqq  \Big\{(\delta_a, \id_H), \, (\mathbf{0}, s) \Big|\ a \in S_A, \; \text{and}  \;  s \in S_H \Big\}.
\]
Then for any element $g=(f,h) \in G$ the corresponding word length 
with respect to $S$ can be computed as 
\[|g|_{\mathrm{S}}=\mathrm{TSP}(\id_H, \supp{f},h)  +\sum_{x\in \supp{f}}|f(x)|_{\mathrm{S}_A}\]
We will call the corresponding metric on $G=A \wr H$ the standard metric.

\begin{cor}\label{cor: CLT for general wr product}
	Let $H$ be a finitely generated acylindrically hyperbolic group and let $A$ be any nontrivial finitely generated group. Assume that $G=A \wr H$ is equipped with a standard metric $d$, and $\mu$ is a symmetric probability measure on $G$ with finite exponential moment such that its projection onto $H$ is a non-elementary measure. Then there exists $\sigma>0$ such that $$\frac{d(\id_G,Z_n) -\ell(\mu,d)n}{\sqrt{n}}$$  converges to $\mathcal{N}(0, \sigma ^2)$ in distribution as $n\xrightarrow[]{}\infty$.
\end{cor}

\begin{proof}
	The proof follows exactly the same outline as the proof of Theorem \ref{CLT acylindrically hyperbolic}. The only modifications are performed in Step 2 in the proof of Theorem \ref{thm: moments of the defect}. These modifications do not change the fact that the upper bounds on the moments of the defect are polynomial in $\log(m+n)$. \footnote{In fact, a polynomial of degree 5 would still work.} Since a polynomial in $\log(m+n)$ satisfies the slow growth assumption, the results of Section \ref{section: MS} remain applicable in this case as well.
	
	As in Step 2 of Theorem \ref{thm: moments of the defect}, let us fix $p >1$. 
	For $x \in H$, denote by $f_1(x) \in A$ the state of the lamp at $x$ after the first $m$ steps of the random walk. Similarly, let $f_2(x)$ denote the change in the state of the lamp at $x$ under the trajectory of the random walk between the times $m+1$ and $m+n$, and let $f_3(x)$ be the final state of the lamp at $x$ at time $m+n$. It is easy to see that $f_3(x)=f_1(x)f_2(x)$. In the notation from the proof of Theorem \ref{thm: moments of the defect}, $L_1$, $L_2$, and $L_3$ are the supports of $f_1$, $f_2$ and $f_3$ respectively. The arguments related to the defect of $\mathrm{TSP}$ and the upper bounds on $N$ and $D$ remain essentially unchanged.
	However, instead of proving an upper bound on the difference $|L_1|+|L_2| -|L_3|$, we now need to obtain the upper bound on \[\sum_{x\in H}|f_1(x)|_{\mathrm{S}_A}+|f_2(x)|_{\mathrm{S}_A} -|f_3(x)|_{\mathrm{S}_A}. \]
	The triangle inequality for the word length implies that $|f_1(x)|_{\mathrm{S}_A}+|f_2(x)|_{\mathrm{S}_A} -|f_3(x)|_{\mathrm{S}_A}$ is always non-negative, and it is easy to see that it can take positive value only if $x \in L_1 \cap L_2$.
	
	Since we have the inequality $ |L_1 \cap L_2| \leq N$, and Step 2 in Theorem \ref{thm: moments of the defect} shows that $N$ is bounded above by a polynomial in $\log (m+n)$, it suffices to prove that under the assumptions in the Step 2, for any  $x \in L_1 \cap L_2$, the difference  $|f_1(x)|_{\mathrm{S}_A}+|f_2(x)|_{\mathrm{S}_A} -|f_3(x)|_{\mathrm{S}_A}$ can be bounded by a polynomial in $\log(m+n)$ that does not depend on $x$. Moreover, it suffices to find an upper bound on $|f_1(x)|_{\mathrm{S}_A}+|f_2(x)|_{\mathrm{S}_A}$. Notice that in Step 2 we established that with overwhelming probability, for every $k=1 ,\ldots, m+n$, the lamp at $x$ may be modified at time $k$ only if the position of the projected random walk $\overline{Z}_{k-1}$ satisfies the inequality $d_H(x,\overline{Z}_{k-1}) \leq C_0 \log(m+n)$. Moreover, this modification consists of multiplying the current state of the lamp at $x$ on the right by an element $a \in A$  such that $|a|_{\mathrm{S}_A} \leq C_0 \log(m+n)$. It remains to notice that by Proposition \ref{lem: uniform progress}, with overwhelming probability, the projected random walk $\overline{Z}_i$ visits any ball of radius $C_0 \log(m+n)$ in the Cayley graph of $H$ at most $2K_0C_0 \log(m+n)+1$ times.  Therefore, we may conclude that for every $x \in L_1 \cap L_2$,  we have $$|f_1(x)|_{\mathrm{S}_A} \leq C_0 \log(m+n)(2K_0C_0 \log(m+n)+1)$$
	and 
	
	$$|f_2(x)|_{\mathrm{S}_A} \leq C_0 \log(m+n)(2K_0C_0 \log(m+n)+1).$$
	
	As a result, we have established the desired upper bound  $$|f_1(x)|_{\mathrm{S}_A}+|f_2(x)|_{\mathrm{S}_A} \leq 2C_0 \log(m+n)(2K_0C_0 \log(m+n)+1),$$
	that does not depend on $x \in L_1 \cap L_2$. 
	The rest of the proof follows the same steps as those in the proof of Theorem \ref{thm: moments of the defect}.
\end{proof}

	\bibliographystyle{alpha}
	\bibliography{biblio.bib}
\end{document}